\documentclass{article}
\usepackage[T2A]{fontenc}
\usepackage[cp1251]{inputenc}
\usepackage[russian,english]{babel}
\usepackage[tbtags]{amsmath}
\usepackage{amsfonts,amssymb,mathrsfs,amscd}
 \usepackage{msb-a}
\JournalName{}
\numberwithin{equation}{section}
\theoremstyle{plain}

\newtheorem{theorem}{Theorem}
\theoremstyle{plain}

\newtheorem{lemma}{Lemma}

\theoremstyle{definition}
\newtheorem{proof}{Proof}

\newtheorem{remark}{Remark}
\newcommand{\mlegendre}[2]{\left(\frac{#1}{#2}\right)}
\renewcommand{\refname}{References}

\renewcommand{\abstract}{\begin{bf} Abstract.\end{bf}}
\begin{document}

\title{Large values of short character sums}
\author{A.\,B.~Kalmynin}
\address{National Research University Higher School of Economics,
Russian Federation, Math Department, International
Laboratory of Mirror Symmetry and Automorphic Forms}
\email{alkalb1995cd@mail.ru}
\date{}
\udk{}

\maketitle
\begin{abstract}
In this paper, we prove that for any $A>0$ there exist infinitely many primes $p$ for which sums of the Legendre symbol modulo $p$ over an interval of length $(\ln p)^A$ can take large values.
\end{abstract}
\begin{fulltext}
\section{Introduction}
The article \cite{GS} contains a variety of propositions on the values of character sums, including the following:

\begin{theorem}

Let $\chi$ be a nonprincipal character to the modulus $q$ and assume that the Riemann hypothesis for the function $L(s,\chi)$ holds. Then for $q\to \infty$ and $\frac{\ln x}{\ln\ln q}\to \infty$ we have

$$\sum_{n\leq x} \chi(n)=o(x)$$

\end{theorem}

Denote by  $\mlegendre{a}{b}$ the Kronecker-Jacobi symbol and let $D$ be a fundamental discriminant, that is, $D \equiv 1 \pmod 4$ and squarefree or $4 \mid D$ and $\frac{D}{4} \equiv 2 \text{ or } 3 \pmod 4$ and is a squarefree number. Then the arithmetical function $\mlegendre{D}{n}$ is a primitive character to the modulus $|D|$ (\cite{Mont}, Theorem 9.13). Therefore, if the generalized Riemann hypothesis is true and $x(D)$ satisfies the condition $$\frac{x(D)}{(\ln |D|)^A} \to +\infty$$ for any $A$, then the sum of $\mlegendre{D}{n}$ over the interval $[1,x(D)]$ is $o(x(D))$ as  $|D|\to \infty$. In the opposite direction, the following theorem is proved:

\begin{theorem}
For all sufficiently large $q$ and any $A>0$ there exists a fundamental discriminant $D$ with $q\leq |D|\leq 2q$ such that the inequality

$$\sum_{n \leq x} \mlegendre{D}{n}\gg_A x$$

holds, where $x=(\frac{1}{3}\ln q)^A$.
\end{theorem}

In this paper we prove a strengthening of the Theorem 2, namely, the following fact:

\begin{theorem}

Let  $A\geq 1$ be an arbitrarily large fixed number, $x\geq x_0(A)$ and $y=(\ln x)^A$. Then there exists a prime number $p$ with $x<p\leq 2x$ such that the inequality

$$\sum_{n\leq y} \mlegendre{n}{p} \gg_A y$$

holds.

\end{theorem}

\begin{remark}

Let $A>0$ and $x$ be large enough. Taking $p \in (x,2x]$ such that

$$\sum_{n\leq (\ln x)^A} \mlegendre{n}{p}\gg_A (\ln p)^A$$

and $\begin{bf}{G}\end{bf}=\mathbb Z/p\mathbb Z, X=\{1,2,\ldots,z\},Y=\{0,z,2z,\ldots,z(z-1)\} \subset \begin{bf}G\end{bf}$, where $z=[(\ln p)^{A/2}]$ we get

$$\sum_{\substack{x \in X \\ y \in Y}} \mlegendre{x+y}{p}=\sum_{n \leq z^2} \mlegendre{n}{p} \geq c_A(\ln p)^A\sim c_A|X||Y|$$

for some positive constant $c_A$. So, if $\mathcal A$ is the set of all quadratic residues in $\mathbb Z/p\mathbb Z$, we have

$$\sum_{\substack{x \in X\\ y \in Y}} \chi_{\mathcal A}(x+y)=\frac{1}{2}|X||Y|+\frac{1}{2}\sum_{\substack{x \in X\\ y \in Y}} \mlegendre{x+y}{p}\geq \frac{1+c_A+o(1)}{2}|X||Y|,$$

where $\chi_{\mathcal A}(\cdot)$ is the characteristic function of the set $\mathcal A$. This inequality, together with the Theorem 2 of the paper \cite{KS} proves that the set of quadratic residues modulo $p$ does not behave like a random subset of $\mathbb Z/p\mathbb Z$ for infinitely many $p$.

\end{remark}

\section{Lemmas and proof of main theorem}

To prove the Theorem 3, we need some auxilary lemmas.

 \begin{lemma}
 
 For some constants $c$ and $c_1>0$  and any nonprincipal character $\chi$ to the modulus $q \leq e^{c_1\sqrt{\ln x}}$ we have
 
 $$\sum_{p\leq x} \chi(p)\ln p=-\delta \frac{x^\beta}{\beta}+O(xe^{-c_1\sqrt{\ln x}}),$$
 
 where $\delta=1$ if $L(s,\chi)$ has a real zero $\beta$ with $\beta>1-\frac{c}{\ln q}$ and $\delta=0$ otherwise.
 \end{lemma}
 \begin{proof}
See \cite{K}, Chapter IX, proof of the Theorem 6.
\end{proof}
 \begin{lemma}
 
 Let $\chi_1$ and $\chi_2$ be two different primitive real characters to the moduli $q_1$ and $q_2$, and suppose that the functions $L(s,\chi_1)$ and $L(s,\chi_2)$ have real zeros $\beta_1$ and $\beta_2$. Then the inequality
 
 $$\min\{\beta_1,\beta_2\}<1-\frac{c_2}{\ln (q_1q_2)}$$
 
 holds for some positive absolute constant $c_2$.
 
 \end{lemma}
\begin{proof}
See \cite{K}, Chapter IX, Theorem 4.
\end{proof}
\begin{lemma}

Let $\chi$ be a real primitive character to the modulus $q$. Then for some absolute constant $c_3>0$ we have

$$L(\sigma,\chi)\neq 0 \text{ for }\sigma>1-\frac{c_3}{\sqrt q \ln^4 q}.$$

\end{lemma}
\begin{proof}
See \cite{K}, Chapter IX, Theorem 3.
\end{proof}

\begin{lemma}

Let $a\neq 0$ be an integer with $a=bc^2$, where $b$ is squarefree and $c$ is a positive integer. Then there exists a primitive quadratic character $\chi_b$ to the modulus $|b^*|$ such that for any odd prime $p$ the equality

$$\mlegendre{a}{p}=\chi_{0,c}(p)\chi_b(p)$$

holds, where

$$b^*=\begin{cases} b, \text{if } b \equiv 1 \pmod 4 \\ 4b, \text{if } b\not\equiv 1 \pmod 4\end{cases}$$

and $\chi_{0,c}$ is the principal character to the modulus $c$.
\end{lemma}

\begin{proof}[of Lemma 4] Using the multiplicativity of Legendre symbol, we obtain

$$\mlegendre{a}{p}=\mlegendre{b}{p}\mlegendre{c^2}{p}.$$

If $p \mid c$, then $\mlegendre{c^2}{p}=0$. On the other hand, if $c$ and $p$ are coprime, then $c^2$ is a quadratic residue modulo $p$, so $\mlegendre{c^2}{p}=1$. Consequently, the equality

$$\mlegendre{c^2}{p}=\chi_{0,c}(p)$$

holds.

Let us now prove that $\mlegendre{b}{p}=\chi_b(p)$ for some primitive charater $\chi_b$ to the modulus $|b^*|$. Note that $b^*$ is a fundamental discriminant. Indeed, if $b \equiv 1 \pmod 4$ then $b^*=b$ is a fundamental discriminant. If, however, $b \not\equiv 1 \pmod 4$, then due to the fact that $b$ is squarefree we have $b \equiv 2 \text{ or } 3 \pmod 4$, thus $b^*=4b$ is also a fundamental discriminant. Hence, $\chi_b(n)=\mlegendre{b^*}{n}$ is a primitive quadratic character to the modulus $|b^*|$ (see \cite{Mont}, Theorem 9.13). On the other hand, $\frac{b^*}{b}$ can take only two values: $1$ and $4$. Therefore, for any odd prime number $p$ we have

$$\chi_b(p)=\mlegendre{b^*}{p}=\mlegendre{b^*/b}{p}\mlegendre{b}{p}=\mlegendre{b}{p},$$

which concludes the proof.

\end{proof}

\begin{lemma}

Let $A\geq 1$, $x\geq x_0(A)$, $M=(\ln x)^{1/3}$, $N=M^{3A}=(\ln x)^A$ and $P(M)=\prod\limits_{q\leq M} q$, where $q$ runs over prime numbers. For any squarefree $c \leq e^{2M}$ we define

$$r_{A,x}(c)=\#\{(a,d): 1\leq a \leq N, d\mid P(M), ad=cy^2 \text{ for some integer $y$}\}.$$

Then we have $r_{A,x}(c) \ll \frac{N\ln N}{M}$ if $c\nmid P(M)$ and $r_{A,x}(c)=r_{A,x}(1)\gg_A N$ if 
$c\mid P(M)$.

\end{lemma}

\begin{proof}[of Lemma 5] If $c \nmid P$, then $c=sr$, where $s \mid P$ and all the prime factors of $r$ are greater than $M$. 









Starting with the pair $(a,d)$ which satisfies the conditions $ad=cy^2$, $a \leq N$ and $d\mid P(M)$, we construct the integer triple  $(d_1,s_1,z)$ in the following way:

$$d_1=\frac{d}{(d,s)}, s_1=\frac{s}{(s,d)}, z=\frac{y}{d_1}.$$

Let us show that $z$ is integer. Indeed, we have $ad=cy^2$, therefore $ad_1=s_1ry^2$, but $(d_1,s_1r)=1$ as all the prime factors of $r$ are greater than $M$ and $d_1$ and $s_1$ are coprime by definition. Hence, $y^2$ is divisible by $d_1$. But $d_1$ is squarefree, so $y$ is also divisible by $d_1$, consequently $z$ is integer. On the other hand, $d_1s_1z^2\leq \frac{N}{r}<\frac{N}{M}$. Indeed,

$$d_1s_1z^2=\frac{s_1y^2}{d_1}=\frac{sy^2}{d}=\frac{cy^2}{dr}=\frac{a}{r} \leq \frac{N}{r}.$$

Furthermore, the pair $(a,d)$ can be recovered from the triple $(s_1,d_1,z)$, due to the fact that

$$a=rs_1d_1z^2 \text{ and } d=\frac{sd_1}{s_1}.$$

Note now that for any $n \leq \frac{N}{M}$ the number of triples $(s_1,d_1,z)$ such that $s_1$ and $d_1$ are squarefree and coprime and $n=s_1d_1z^2$ does not exceed the number of divisors of $n$. Indeed, from these conditions we deduce that $s_1d_1$ is a squareefree number, so it is equal to the squarefree part of $n$. Hence, $z$ is uniquely defined and the number of possible pairs $(s_1,d_1)$ is equal to the number of divisiors of the squarefree part of $n$. Therefore, $r_{A,x}(c)$ does not exceed the number of triples satisfying these conditions, which is less than or equal to

$$\sum_{n \leq \frac{N}{M}}d(n) \ll \frac{N\ln N}{M},$$

which completes the proof in this case.\\
In the case when $c \mid P(M)$, with any pair $(a,d)$ satisfying the conditions $1\leq a\leq N$, $d\mid P(M)$, $ad=cy^2$ for some $y \in \mathbb Z$ we associate the pair $\left(a,\frac{cd}{(c,d)^2}\right)$. It is easy to see that $\frac{cd}{(c,d)^2} \mid P(M)$ and

$$\frac{acd}{(c,d)^2}=\frac{ac^2y^2}{(c,d)^2}=z^2,$$

where $z=\frac{cy}{(c,d)}$ is integer. Conversely, any pair $(a,d)$ with $1\leq a \leq N$, $d\mid P(M)$, $ad=z^2$ for some $z \in \mathbb Z$ corresponds to the pair $\left(a,\frac{cd}{(c,d)^2}\right)$. Then $\frac{cd}{(c,d)^2} \mid P(M)$ and

$$\frac{acd}{(c,d)^2}=\frac{cz^2}{(c,d)^2}=cy^2$$

for some integer $y$, as $z$ is divisible by $d$. These maps are mutually inverse, due to the fact that for $D=\frac{cd}{(c,d)^2}$ we have $(c,D)=\frac{c}{(c,d)}$ and so $\frac{cD}{(c,D)^2}=d$. Therefore $r_{A,x}(1)=r_{A,x}(c)$. Now note that for any $M-$smooth number $a\leq N$ there exists $d \mid P(M)$ with $ad=y^2$ (it suffices to choose the squarefree part of $a$ instead of $d$). Consequently,

$$r_{A,x}(1)\geq \Psi(N,M)=(\rho(3A)+o(1))N\geq \exp(-3A(\ln (3A)+\ln\ln A))N$$

for all large enough $A$, where $\rho$ is the Dickman function (see \cite{Hild}, pp 4-5), which completes the proof of lemma.
\end{proof}

\begin{lemma}

Assume that $x>e^4$ and let $\beta$ be a real number with $0<1-\beta\leq \frac{2}{\ln x}$. Then we have

$$x-\frac{(2x)^\beta}{\beta}+\frac{x^{\beta}}{\beta} \geq \frac{(1-\beta)x\ln x}{4e^2}.$$

\end{lemma}

\begin{proof}[of Lemma 6] As $\beta<1$, we have $(2x)^\beta<2x^\beta$. Thus,

$$x-\frac{(2x)^\beta}{\beta}+\frac{x^\beta}{\beta}>x-\frac{x^\beta}{\beta}.$$

On the other hand,

$$x-\frac{x^\beta}{\beta}=\int_\beta^1 \frac{x^s}{s^2}(s\ln x-1)ds.$$

For $\beta\leq s\leq 1$ the inequalities

$$\frac{x^s}{s^2} \geq x^\beta \geq x^{1-2/\ln x}=\frac{x}{e^2}$$

hold and

$$s\ln x-1\geq \beta\ln x-1\geq  \ln x\left(1-\frac{2}{\ln x}\right)-1=\ln x-3 \geq \frac{\ln x}{4},$$

consequently,

$$x-\frac{x^\beta}{\beta} \geq \int_\beta^1 \frac{x\ln x}{4e^2}ds=\frac{(1-\beta)x\ln x}{4e^2}.$$

The lemma is proved.

\end{proof}

\begin{remark}
For $1-\frac{2}{\ln x} \geq \beta \geq \frac12$ we have

$$\frac{(2x)^\beta}{\beta} \leq 4x^{\beta} \leq \frac{4}{e^2}x,$$

hence

$$x-\frac{(2x)^\beta}{\beta} \geq x\left(1-\frac{4}{e^2}\right) \geq \frac{x}{e^2}.$$

Therefore, for all $\frac12 \leq \beta<1$ and $x>e^4$ the inequality

$$x-\frac{(2x)^\beta}{\beta}+\frac{x^\beta}{\beta} \geq \frac{(1-\beta)x}{e^2}$$\

holds.
\end{remark}

Let us now prove the Theorem 3.

\begin{proof}[of Theorem 3] Fix some $A\geq1$, choose large enough $x$ and, following the notation of Lemma 5, take $M=(\ln x)^{1/3}$, $N=M^{3A}=(\ln x)^A$ and $P(M)=\prod\limits_{q \leq M} q$. Let us introduce the following notations

$$w_p(M)=\prod_{q \leq M} \left(1+\mlegendre{q}{p}\right)$$

and

$$S(p,N)=\sum_{n \leq N} \mlegendre {n}{p}.$$

It is easy to see that $w_p(M) \geq 0$ for any prime $p$. To prove the Theorem 3 it suffices to show that there exists at least one $p \in (x,2x]$ such that $S(p,N)\gg_A N$.

Consider the following sums:

$$S_0(x)=\sum_{x<p\leq 2x} w_p(M)\ln p$$

and

$$S_1(x,A)=\sum_{x<p\leq 2x} w_p(M)S(p,N)\ln p.$$

By the nonnegativity of $w_p(M)$, it is enough to show that the sum $S_0(x)$ is positive and

$$S_1(x,A) \gg_A NS_0(x).$$

We will prove the positivity of $S_0(x)$ first.

Expanding the brackets in the definition of $w_p(M)$, we obtain

$$w_p(M)=1+\sum_{\substack{d \mid P(M) \\ d>1}} \mlegendre{d}{p}.$$

Therefore,

$$S_0(x)=\sum_{x<p \leq 2x} \ln p+\sum_{\substack{d \mid P(M) \\ d>1}} \sum_{x<p\leq 2x} \mlegendre{d}{p}\ln p.$$

Let us choose a positive $c_4$ such that $c_4<0.1\min\{c,c_1,c_2,c_3\}$, where $c,c_1,c_2,c_3$ are the absolute constants from lemmas 1, 2 and 3. Then, by the lemmas 1 and 4, for any squarefree integer $1<f \leq e^{2M}$ there exists a primitive real character $\chi_f$ to the modulus $f^*$ such that

\begin{equation}
\label{1}
 \sum_{x<p\leq 2x} \mlegendre{f}{p}\ln p=-\delta_f\frac{(2x)^\beta-x^\beta}{\beta}+O(xe^{-c_4\sqrt{\ln x}}),
 \end{equation}
 
 where $\delta_f=1$ if the function $L(s,\chi_f)$ has a real zero $\beta$ with $\beta>1-\frac{c_4}{\ln f}$ and $\delta_f=0$ otherwise.
 
 On the other hand, due to the Lemma 2 there exists at most one squarefree $f \leq e^{2M}$ such that $\delta_f=1$. In the case when such a number exists, we denote it by $g$ and the corresponding real zero by $\beta_g$. Using lemma 3, for large enough $x$ we obtain the inequality
 
 $$\beta_g\leq 1-\frac{c_4}{\sqrt{g}\ln^4 g}\leq 1-\frac{1}{g}\leq 1-e^{-2M}.$$
 
Using the equality ~(\ref{1}) and relations $P(M)\leq e^{2M} \ll e^{0.4c_4\sqrt{\ln x}}$, we get the following expression for $S_0(x)$:
 
 $$S_0(x)=x-\delta_1\frac{(2x)^{\beta_g}-x^{\beta_g}}{\beta_g}+O(xe^{-0.5c_4\sqrt{\ln x}}),$$
 
where $\delta_1=1$, if $g$ exists and divides $P(M)$ and $\delta_1=0$ otherwise.
 
Due to the inequality $1-\beta_g\geq 1-e^{-2M}$ and Lemma 6, we obtain

$$x-\delta_1\frac{(2x)^{\beta_g}-x^{\beta_g}}{\beta_g} \gg x(1-\beta_g)\gg xe^{-2M} \gg xe^{-0.4c_4\sqrt{\ln x}}.$$

Consequently,

$$S_0(x)\sim x-\delta_1\frac{(2x)^{\beta_g}-x^{\beta_g}}{\beta_g},$$

and, in particular, $S_0(x)>0$. Let us now estimate $S_1(A,x)$. Multiplying the expressions for $w_p(M)$ and $S(p,N)$, we get

$$w_p(M)S_p(N)=\sum_{\substack{d \mid P(M) \\ n\leq N}} \mlegendre{dn}{p}.$$

If $d \mid P(M)$ and $n \leq N$, then $dn \leq e^{2M}$. If $b$ is a squarefree part of $dn$, then for any prime $x<p\leq 2x$ we have $\mlegendre{dn}{p}=\mlegendre{b}{p}$. In view of this, for any prime $x<p\leq 2x$ the equality

$$w_p(M)S_p(N)=\sum_{b\leq e^{2M}} r_{A,x}(b)\mlegendre{b}{p}$$

holds. Taking the sum of this expression over all prime numbers from the interval $(x,2x]$, we get

$$S_1(A,x)=\sum_{b \leq e^{2M}} r_{A,x}(b)\sum_{p<x\leq 2x} \mlegendre{b}{p}\ln p=$$
$$=r_{A,x}(1)\sum_{x<p\leq 2x}\mlegendre{1}{p}+\delta_2r_{A,x}(g)\sum_{x<p\leq 2x}\mlegendre{g}{p}+O\left(\sum_{\substack{b\neq 1,g\\ b\leq e^{2M}}} r_{A,x}(b)xe^{-c_4\sqrt{\ln x}}\right)=$$
$$=xr_{A,x}(1)-\delta_2\frac{(2x)^{\beta_g}-x^{\beta_g}}{\beta_g}r_{A,x}(g)+O(xe^{-0.5c_4\sqrt{\ln x}}),$$

where $\delta_2=1$, if $g$ exists and $\delta_2=0$ otherwise.

If $g$ doesn't exist, then $\delta_1=\delta_2=0$, so $$S_1(A,x)=xr_{A,x}(1)+O(xe^{-0.5c_4\sqrt{\ln x}})\gg_A Nx$$ by the Lemma 5. But we also have $S_0(A,x)\asymp x$ and so

$$\frac{S_1(A,x)}{S_0(x)}\gg_A N,$$

which conludes the proof in this case.

If  $g$ exists and does not divide $P(M)$, then $\delta_2=1$ and $\delta_1=0$. Therefore,

$$S_0(x) \asymp x$$

and

$$S_1(A,x)=xr_{A,x}(1)-\frac{(2x)^{\beta_g}-x^{\beta_g}}{\beta_g}r_{A,x}(g)+O(xe^{-0.5c_4\sqrt{\ln x}}).$$

By the Lemma 5, $r_{A,x}(1)\gg_A N$ and $r_{A,x}(g) \ll \frac{N\ln N}{M}$, hence

$$S_1(A,x)=r_{A,x}(1)x\left(1+O\left(\frac{\ln N}{M}\right)\right).$$

Consequently,

$$\frac{S_1(A,x)}{S_0(x)}\gg_A N.$$

It remains to consider the case when $g$ exists and divides $P(M)$. Then we have $\delta_1=\delta_2=1$, thus

$$S_0(x) \asymp x-\frac{(2x)^{\beta_g}-x^{\beta_g}}{\beta_g}$$

and

$$S_1(A,x)=xr_{A,x}(1)-\frac{(2x)^{\beta_g}-x^{\beta_g}}{\beta_g}r_{A,x}(g)+O(xe^{-0.5c_4\sqrt{\ln x}}).$$

Due to the Lemma 5, $r_{A,x}(1)=r_{A,x}(g)\gg_A N$. Using Lemma 6, we get

$$\frac{S_1(A,x)}{S_0(x)}=r_{A,x}(1)+O\left(\frac{e^{-0.5c_4\sqrt{\ln x}}}{1-\beta_g}\right)=(\rho(3A)+o(1))N+O\left(e^{-0.1c_4\sqrt{\ln x}}\right)\gg_A N,$$

which completes the proof of our theorem.

\end{proof}

\section*{Acknowledgements}

The author thanks Maxim Aleksandrovich Korolev and Sergei Vladimirovich Konyagin for the useful comments and discussion.

The author is partially supported by Laboratory of Mirror Symmetry NRU HSE, RF Government grant, ag. \textnumero 14.641.31.0001, the Simons Foundation and the Moebius Contest Foundation for Young Scientists and by the Program of the Presidium of the Russian Academy of Sciences \textnumero01 'Fundamental Mathematics and its Applications' under grant PRAS-18-01.

\end{fulltext}
\end{document}